\documentclass[11pt]{amsart}
\usepackage{geometry} 
\def\R{\mathbb{R}}

\def\id{\mathbf{1}}

\def\bA{\mathcal{A}}

\def\dv{{\hbox{div}\,}}
\def\curl{{\hbox{curl}\,}}
\def\supp{{\hbox{supp}\,}}

\def\bx{\bold x}

\def\bu{\bold u}
\def\bv{\bold v}
\def\bU{\bold U}
\def\bbU{\mathbb{U}}
\def\bV{\bold V}

\def\ba{\bold a}
\def\bbe{\bold b}

\def\bA{\bold A}
\def\bC{\bold C}

\def\bQ{\bold Q}
\def\bF{\bold F}

\def\bcero{\bold 0}

\numberwithin{equation}{section}

\newtheorem{theorem}{Theorem}

\newtheorem{lemma}{Lemma}
\newtheorem{proposition}{Proposition}
\newtheorem{problem}{Problem}

\title[Non-linear conductivity]{Weak limits in non-linear conductivity}
\author{Pablo Pedregal}
\address{ETSI Industriales, Universidad de Castilla-La Mancha, 13071 Ciudad Real, Spain}
\email{pablo.pedregal@uclm.es}
\subjclass[2000]{}
\thanks{
E.T.S. Ingenieros Industriales. Universidad de Castilla La Mancha.
Campus de Ciudad Real (Spain). Research supported by
MTM2010-19739 of the MCyT (Spain).
 e-mail:{\tt pablo.pedregal@uclm.es
}}

\begin{document}
\maketitle
\begin{abstract}
The problem of characterizing weak limits of sequences of solutions for a non-linear diffusion equation of $p$-laplacian type is addressed. It is formulated in terms of certain moments of underlying Young measures associated with main fields corresponding to the non-linear equation. We show two main results. One is a necessary condition; the other one a sufficient counterpart. Unlike the linear case, those two conditions do not match. 
\end{abstract}
\section{Introduction}
We are concerned here with a typical optimal design problem in conductivity. Such problems are by now well understood both from an analytical perspective as well as a practical and computational viewpoints. There is a great amount of literature on this topic from all important and relevant points of view. Many deep results have been established in the context of homogenization over the years, particularly in \cite{AllaireA}, \cite{Cherkaev}, \cite{MuratTartarC}, \cite{MuratTartar}, \cite{MuratTartarA}, \cite{TartarK}, \cite{TartarM}. From a more applied-oriented viewpoint see for instance \cite{BendsoeA}. 
See also \cite{miltonZ} for a comprehensive reference to many of the mechanical implications of homogenization theory.  \cite{LiptonD} treats the fundamental and difficult gradient constraints.

The situation that we would like to address corresponds to a non-linear conductivity equation.  Although more general situations can be treated under suitable structural assumptions, to make computations more explicit we will stick to a $p(=4)$-laplacian type equation 
\begin{equation}\label{nolin}
\dv[(\chi(\bx)\alpha_1+(1-\chi(\bx))\alpha_0)|\nabla u(\bx)|^2\nabla u(\bx)]=0\hbox{ in }\Omega,\quad u=u_0\hbox{ on }\partial\Omega.
\end{equation}
The design domain $\Omega\subset\R^N$ is supposed to be a bounded, Lipschitz domain; $u_0$ furnishes the boundary datum; $\alpha_i$ are two non-negative constants $\alpha_1>\alpha_0$ associated with each of the non-linear materials at our disposal. The design variable $\chi(\bx)$ is a characteristic (binary) variable indicating where in $\Omega$ to place each of the two materials in fixed global proportions
$$
\int_\Omega\chi(\bx)\,d\bx=r|\Omega|,\quad r\in(0, 1).
$$
The coefficient $\chi(\bx)\alpha_1+(1-\chi(\bx))\alpha_0$ yields the non-linear conductivity coefficient of the mixture. Note that (\ref{nolin}) admits a unique solution $u\in W^{1, 4}(\Omega)$ if $u_0$ is the trace on $\partial\Omega$ of a given $u_0\in W^{1, 4}(\Omega)$. 

The choice of $\chi$ is made according to a certain functional cost. Since we would like to focus on the effect of the underlying non-linear state equation (\ref{nolin}), at this initial contribution we will assume the  simplest functional cost so that it does not interfere with the non-linear nature of the state equation in the relaxation process, namely,
$$
I(\chi)=\int_\Omega\bF(\bx)\cdot\nabla u(\bx)\,d\bx,
$$
where $u\in W^{1, 4}(\Omega)$ is the solution of (\ref{nolin}). The field $\bF$ is also given, and belongs to $L^{4/3}(\Omega; \R^N)$. If $u_j$ is the corresponding solution of (\ref{nolin}) for $\chi_j$, and $u_j\rightharpoonup u$ in $W^{1, 4}(\Omega)$, then it is clear that 
$$
I(\chi_j)\to\int_\Omega\bF(\bx)\cdot\nabla u(\bx)\,d\bx.
$$
In this way, relaxation amounts to characterizing possible weak limits of solutions of (\ref{nolin}). Understanding these weak limits under the non-linear conductivity law is the main purpose of this contribution. 

More specifically, we would like to address the issue of characterizing weak limits of sequences of pairs $(\bU_j, \bV_j)$ under the constraints
\begin{enumerate}
\item $\curl\bU_j=0$, $\dv\bV_j=0$ in $\Omega$ for all $j$;
\item $\bV_j=(\chi\alpha_1+(1-\chi)\alpha_0)|\bU_j|^2\bU_j$;
\item $|\{\chi=1\}|=t|\Omega|$.
\end{enumerate}
In addition, one can keep in mind to demand $u_j=u_0$ on $\partial\Omega$ if $\bU_j=\nabla u_j$. 

Our two main results provide a necessary, and a sufficient condition that, unfortunately, do not match. 
\begin{theorem}
Suppose $(\bU, \bV)$ is a pair of (constant) fields which is a weak limit (in $L^2(\Omega; \R^2\times\R^2)$) of a sequence as just indicated, corresponding to relative volume fraction $t$. Then
$$
\gamma|\bU|^4\le\bV\cdot\bU,
$$
for
$$
\gamma=\frac{\alpha_1\alpha_0}{((1-t)\alpha_1^{1/3}+t\alpha_0^{1/3})^3}.
$$
\end{theorem}
For the sufficient condition, consider the function
$$
\Psi_{(t, \bU)}(\bx)=(\alpha_0|t\bx+\bU|^2(t\bx+\bU)+\alpha_1|(1-t)\bx-\bU|^2((1-t)\bx-\bU))\cdot\bx,\quad \bx\in\R^N,
$$
and the related map
$$
\Phi_{(t, \bU)}:\R^N\mapsto\R^N,\quad 
\Phi_{(t, \bU)}(\bx)=t\alpha_1|\bU-(1-t)\bx|^2(\bU-(1-t)\bx)+(1-t)\alpha_0|\bU+t\bx|^2(\bU+t\bx).
$$

\begin{theorem}
The triplet $(t, \bU, \bV)$ is reachable by lamination if and only if
$$
\bV\in\Phi_{(t, \bU)}(\bC(t, \bU)),
$$
where
$$
\bC(t, \bU)=\{\bbU\in\R^N: \Psi_{(t, \bU)}(\bbU)\le0\}.
$$
\end{theorem}
For the proof of these two results, we briefly review the linear situation in a way that can be extended for the non-linear case, at least to derive necessary and sufficient conditions. 
This goes directly into formulating the problem in terms of moments of underlying Young measures (\cite{BalderD}, \cite{MullerE}, \cite{PedregalI}, \cite{PedregalAB}, \cite{ValadierA}, \cite{YoungB}). Our point of view is similar to the one described in \cite{boussaidpedregal}, \cite{PedregalZ}.
A fundamental role is played in this regard by the classical div-curl lemma (\cite{MuratC}, \cite{TartarA}, \cite{TartarM}). 

\section{The linear, isotropic, homogenous case revisited}
As a preliminary step to introduce the ideas for the non-linear situation, let us consider the standard problem of the description of weak limits for the linear case, namely
$$
\dv[(\alpha_1\chi+\alpha_0(1-\chi)) \nabla u]=0\hbox{ in }\Omega, u=u_0\hbox{ on }\partial\Omega,
$$
where $\chi\equiv\chi(\bx)$ is a certain characteristic function in $\Omega$ (which is supposed to be a regular, bounded domain in $\R^N$), further restricted by asking
$$
\int_\Omega\chi(\bx)\,d\bx=r|\Omega|,\quad r\in(0, 1).
$$
If we focus on the field
$$
\bV(\bx)=(\alpha_1\chi(\bx)+\alpha_0(1-\chi(\bx))) \nabla u(\bx)
$$
and look at the pair $(\bU, \bV)$, where $\bU=\nabla u$, then we know, among other things, that
\begin{enumerate}
\item $\dv\bV=0$, $\curl\bU=0$ in $\Omega$;
\item $(\bU(\bx), \bV(\bx))\in\Lambda_1\cup\Lambda_0$ where
$$
\Lambda_i=\{(\bu, \bv)\in\R^N\times\R^N: \bv=\alpha_i\bu\},
$$
a linear manifold;
\item  the measure of the set
$$
\{\bx\in\Omega: (\bU(\bx), \bV(\bx))\in\Lambda_1\}
$$
is $r|\Omega|$. 
\end{enumerate}
The first two constraints are local in $\Omega$, while the third one, and the Dirichlet boundary condition for $u$, are global. 

Assume we now have a sequence $\{\chi_j\}$ of characteristic functions as above. We would therefore have a sequence of pairs $\{(\bU_j, \bV_j)\}$ just as we have discussed. We can consider the corresponding Young measure
$\{\nu_\bx\}_{\bx\in\Omega}$. The condition on the manifolds $\Lambda_i$ becomes a condition on the support of each $\nu_\bx$, namely
$$
\supp(\nu_\bx)\subset\Lambda_1\cup\Lambda_0
$$
for a.e. $\bx\in\Omega$, and so we can write
$$
\nu_\bx=t(\bx)\nu_{1, \bx}+(1-t(\bx))\nu_{0, \bx},\quad t(\bx)\in[0, 1],\supp\nu_{i, \bx}\subset\Lambda_i, i=1, 0.
$$
Since we are going to focus on the local properties, all of our main discussion takes place around a.e. $\bx\in\Omega$, where $\bx$ acts as a parameter. To lighten a bit the notation, we are going to drop the $\bx$-dependence, bearing in mind that all objects occurring throughout our discussion depend on $\bx$. 

Since the weak limit (in $L^2(\Omega; \R^N)$) $(\bU, \bV)$ of $(\bU_j, \bV_j)$ is precisely the first moment of $\nu$, we have
$$
\bU=t\bU_1+(1-t)\bU_0,\quad \bV=t\alpha_1\bU_1+(1-t)\alpha_0\bU_0,
$$
where
$$
\bU_i=\int_{\R^N}\bu\,d\mu_i(\bu),
$$
and $\mu_i$ is just the projection of $\nu_i$ onto the first copy of $\R^N$. The differential information given through the restrictions $\dv\bV_j=0$, $\curl\bU_j=0$ in $\Omega$, can also be transformed into information for the underlying Young measure through the classic div-curl lemma. Indeed, it is well-known that
$$
\bU\cdot\bV=t\alpha_1 q_1+(1-t)\alpha_0 q_0,\quad q_i=\int_{\R^N}|\bu|^2\,d\mu_i(\bu),\quad i=1, 0.
$$
We therefore are led to consider the following first important problem.

\begin{problem}\label{probunoo}
Find a characterization of the set of triplets $(t, \bU, \bV)\in[0, 1]\times\R^N\times\R^N$ so that 
\begin{gather}
\bU=t\bU_1+(1-t)\bU_0,\quad \bV=t\alpha_1\bU_1+(1-t)\alpha_0\bU_0,\nonumber\\
\bU\cdot\bV=t\alpha_1 q_1+(1-t)\alpha_0 q_0,\nonumber
\end{gather}
and there are probability measures $\mu_i$ supported in $\R^N$ in such a way that
\begin{equation}\label{jensen}
\int_{\R^N}(\bu, |\bu|^2)\,d\mu_i(\bu)=(\bU_i, q_i),\quad i=1, 0.
\end{equation}
\end{problem}
In this particular situation we do know the answer to this problem. It was fully examined in \cite{boussaidpedregal}, within this same approach, and it is also a consequence of more general ideas in homogenization (\cite{TartarK}, \cite{TartarM}). We seek to provide a slightly different treatment that might be extended to the non-linear case. The approach in those other references cannot be done so. What is true, at any rate, is that $q_i\ge|\bU_i|^2$ is the only information that can be drawn from (\ref{jensen}). 

One of the most striking differences with respect to a non-linear situation relates to the identity
\begin{equation}\label{uve}
\bV=t\alpha_1\bU_1+(1-t)\alpha_0\bU_0
\end{equation}
in the statement of Problem \ref{probunoo}. This will, by no means, be so in a non-linear scenario. In trying to adapt to the difficulties we may encounter in such a non-linear situation, let us ignore this information, and proceed to use the other pieces of information, namely,
$$
\bU=t\bU_1+(1-t)\bU_0,\quad \bU\cdot\bV=t\alpha_1 q_1+(1-t)\alpha_0 q_0,\quad q_i\ge|\bU_i|^2.
$$
Instead of using these two equality restrictions, let us introduce two new free variables $\bbU\in\R^N$, and $c\in\R$, by putting
$$
\bU_1=\bU-(1-t)\bbU, \bU_0=\bU+t\bbU,\quad q_1=\frac1{\alpha_1}(\bU\cdot\bV-(1-t)c),
q_0=\frac1{\alpha_0}(\bU\cdot\bV+tc),
$$
and then the inequalities become
$$
\frac1{\alpha_1}(\bU\cdot\bV-(1-t)c)\ge|\bU-(1-t)\bbU|^2,\quad 
\frac1{\alpha_0}(\bU\cdot\bV+tc)\ge|\bU+t\bbU|^2.
$$
By eliminating $c$, we finally have
\begin{equation}\label{elipse}
t\alpha_1|\bU-(1-t)\bbU|^2+(1-t)\alpha_0|\bU+t\bbU|^2\le\bU\cdot\bV.
\end{equation}
What are the conditions on the triplet $(t, \bU, \bV)$ so that there is a vector $\bbU$ complying with this inequality? By considering the function
$$
\Phi(\bbU)=t\alpha_1|\bU-(1-t)\bbU|^2+(1-t)\alpha_0|\bU+t\bbU|^2
$$
and computing its absolute minimum at 
$$
\bbU=\frac{\alpha_1-\alpha_0}{(1-t)\alpha_1+t\alpha_0}\bU,
$$
substituting in $\Phi$, we find the condition, after some algebra,
\begin{equation}\label{restriccion}
\frac{\alpha_1\alpha_0}{(1-t)\alpha_1+t\alpha_0}|\bU|^2\le\bU\cdot\bV.
\end{equation}
The feasible vectors $\bbU$ are then those in region (\ref{elipse}) which represents a certain ellipsoid in $\R^N$. The final issue is what among vectors $\bbU$ in this ellipsoid can, in addition, comply with (\ref{uve}) for a given $\bV$ which together with the pair $(t, \bU)$ satisfies (\ref{restriccion}). In the linear situation this is not difficult to check because (\ref{restriccion}) is a linear constraint on $\bbU$
\begin{equation}\label{lineall}
\bV=t\alpha_1(\bU-(1-t)\bbU)+t\alpha_0(\bU+t\bbU).
\end{equation}
The intersecion of the ellipsoid (\ref{elipse}) with this linear constraint must be non-empty. 

This whole process can be made completely explicit, in this linear situation, if we incorporate in our discussion from the very beginning the condition (\ref{lineall}) to eliminate vector $\bV$, and describe everything in terms of $(t, \bU, \bbU)$. Note that  $\bbU=\bU_1-\bU_0$ is related to the direction of lamination, and
\begin{gather}
\bU_1=\bU-(1-t)\bbU,\quad \bU_0=\bU+t\bbU,\nonumber\\ 
\bV=t\alpha_1(\bU-(1-t)\bbU)+(1-t)\alpha_0(\bU+t\bbU).\nonumber
\end{gather}
For arbitrary $(t, \bU, \bbU)$, these formulae provide vectors complying with the two first equality constraints in Problem (\ref{probunoo}). Let us focus on the other two conditions. It is clear that the one involving $\mu_i$ amounts to having $q_i\ge|\bU_i|^2$, and the one related to the inner product can then be expressed as the inequality
$$
\bU\cdot(t\alpha_1(\bU-(1-t)\bbU)+(1-t)\alpha_0(\bU+t\bbU)\ge t\alpha_1|\bU-(1-t)\bbU|^2+(1-t)\alpha_0|\bU+t\bbU|^2.
$$
This inequality only involves the three set of variables $(t, \bU, \bbU)$ knowing that $\bV$ is given by the expression above in terms of these three. 
Going through the algebra in this inequality with a bit of care, we obtain
\begin{equation}\label{condicion}
\frac{(1-t)\alpha_1+t\alpha_0}{\alpha_1-\alpha_0}|\bbU|^2\le \bU\cdot\bbU.
\end{equation}
\begin{theorem}\label{linear}
A triplet $(t, \bU, \bV)\in[0, 1]\times\R^N\times\R^N$ is a weak limit of a sequence $(\chi_j, \nabla u_j, \bV_j)$ complying with
\begin{gather}
u_j(\bx)=\bU\cdot\bx\hbox{ on }\partial\Omega,\quad \dv\bV_j=0\hbox{ in }\Omega,\nonumber\\
\bV_j(\bx)=\alpha_1\chi_j(\bx)+\alpha_0(1-\chi_j(\bx))\nabla u_j\hbox{ in }\Omega,\nonumber\\
\chi_j(\bx)\in\{1, 0\},\quad \int_\Omega\chi_j(\bx)\,d\bx\to|\Omega|t.\nonumber
\end{gather}
if and only if 
$$
\bV=t\alpha_1(\bU-(1-t)\bbU)+(1-t)\alpha_0(\bU+t\bbU)
$$
and $(t, \bU, \bbU)$ satisfies (\ref{condicion}).
\end{theorem}
\begin{proof}
The necessity part has already been described in the above discussion as a way to find the constraint (\ref{condicion}). Let us look at sufficiency.

We proceed in two steps. Suppose first that (\ref{condicion}) is an exact equality, i.e. 
\begin{equation}\label{igualdad}
\frac{(1-t)\alpha_1+t\alpha_0}{\alpha_1-\alpha_0}|\bbU|^2=\bU\cdot\bbU.
\end{equation}
Set, then, as above
$$
\bU_1=\bU-(1-t)\bbU,\quad \bU_0=\bU+t\bbU. 
$$
It is elementary to check then that the pairs $(\bU_1, \alpha_1\bU_1)$ and $(\bU_0, \alpha_0\bU_0)$ are such that
\begin{equation}\label{laminado}
(\bU_1-\bU_0)\cdot (\alpha_1\bU_1-\alpha_0\bU_0)=0.
\end{equation}
Indeed, this last identity is another way of writing (\ref{igualdad}). But (\ref{laminado}) is exactly the requirement to build a first-order laminate with normal direction precisely given by $\bbU$. 

Consider next the general situation. The constraint (\ref{condicion}) represents a certain ball in $\bbU$ centered at a multiple of $\bU$ and radius also depending on $\bU$. The important feature is that, for given $(t, \bU$), it represents a convex set for $\bbU$ whose boundary corresponds to (\ref{igualdad}), i.e. our step one above. In particular, if $\bbU$ lies in the interior of this convex set, it can be decomposed as a convex combination 
\begin{equation}\label{dec}
\bbU=r\bbU_1+(1-r)\bbU_0
\end{equation}
where each $\bbU_i$ belongs to the boundary of the circle, and $r\in(0, 1)$. By step one, each triplet $(t, \bU, \bbU_i)$ corresponds to a first-order laminate with volume fraction $t$, lamination direction $\bbU_i$, and corresponding $\bV_i$ given by
$$
\bV_i=t\alpha_1(\bU-(1-t)\bbU_i)+(1-t)\alpha_0(\bU+t\bbU_i),\quad i=1, 0.
$$
But because vector $\bU$ is the same for both cases, $i=1, 0$, we can further laminate $(\bU, \bU)$, and $(\bV_1, \bV_0)$ with relative volume fraction $r$, and lamination direction orthogonal to $\bV_1-\bV_0$. This second-order laminate will have a corresponding vector $\bV$ given by the convex combination
$$
r\bV_1+(1-t)\bV_0=t\alpha_1(\bU-(1-t)\bbU)+(1-t)\alpha_0(\bU+t\bbU)=\bV.
$$
\end{proof}
Note that because the constraint (\ref{condicion}) represents a ball through the origen, the decomposition (\ref{dec}) can always be taken as
$$
\bbU=r[(1/r)\bbU]+(1-r)\bcero,\quad \bbU_1\parallel\bbU, \bbU_0=\bcero.
$$

\section{Reformulation and Young measures for the non-linear situation}
Consider the problem
$$
\dv[(\alpha_1\chi|\nabla u|^2+\alpha_0(1-\chi)|\nabla u|^2) \nabla u]=0\hbox{ in }\Omega, u=u_0\hbox{ on }\partial\Omega,
$$
where $\chi\equiv\chi(\bx)$ is a certain characteristic function in $\Omega$, which is supposed to be a regular, bounded domain in $\R^N$, further restricted by asking
$$
\int_\Omega\chi(\bx)\,d\bx=r|\Omega|,\quad r\in(0, 1).
$$
If we focus on the field
$$
\bV(\bx)=(\alpha_1\chi(\bx)|\nabla u(\bx)|^2+\alpha_0(1-\chi(\bx))|\nabla u(\bx)|^2) \nabla u(\bx)
$$
and look at the pair $(\bU, \bV)$, where $\bU=\nabla u$ then we know, among other things, that
\begin{enumerate}
\item $\dv\bV=0$, $\curl\bU=0$ in $\Omega$;
\item $(\bU(\bx), \bV(\bx))\in\Lambda_1\cup\Lambda_0$ where
$$
\Lambda_i=\{(\bu, \bv)\in\R^N\times\R^N: \bv=\alpha_i|\bu|^2\bu\},
$$
a non-linear manifold;
\item  the measure of the set
$$
\{\bx\in\Omega: (\bU(\bx), \bV(\bx))\in\Lambda_1\}
$$
is $r|\Omega|$. 
\end{enumerate}
The first two constraints are local in $\Omega$, while the third one and the Dirichlet boundary condition for $u$, are global just as in the linear counterpart. 

Assume we have a sequence $\{\chi_j\}$ of characteristic functions as above. We would therefore have a sequence of pairs $\{(\bU_j, \bV_j)\}$ just as we have discussed. We can consider the corresponding Young measure
$\{\nu_\bx\}_{\bx\in\Omega}$. The condition on the manifolds $\Lambda_i$ becomes a condition on the support of each $\nu_\bx$, namely
$$
\supp(\nu_\bx)\subset\Lambda_1\cup\Lambda_0
$$
for a.e. $\bx\in\Omega$, and so we can write
$$
\nu_\bx=t(\bx)\nu_{1, \bx}+(1-t(\bx))\nu_{0, \bx},\quad t(\bx)\in[0, 1],\supp\nu_{i, \bx}\subset\Lambda_i, i=1, 0.
$$
As we did earlier, we are going to drop the $\bx$-dependence, bearing in mind that all objects occurring throughout our discussion depend on $\bx$. In addition, because of our definition of the two manifolds, we can also write
$$
\int_{\Lambda_i}\Phi(\bu, \bv)\,d\nu_i(\bu, \bv)=\int_{\R^N}\Phi(\bu, |\bu|^2\bu)\,d\mu_i(\bu)
$$
where $\mu_i$ is just the projection of $\nu_i$ onto the first copy of $\R^N$, and $\Phi$ is an arbitrary continuous function. 

Since the weak limit $(\bU, \bV)$ of $(\bU_j, \bV_j)$ is precisely the first moment of $\nu$, we have
$$
\bU=t\bU_1+(1-t)\bU_0,\quad \bV=t\alpha_1\bV_1+(1-t)\alpha_0\bV_0,
$$
where
$$
\bU_i=\int_{\R^N}\bu\,d\mu_i(\bu),\quad
\bV_i=\int_{\R^N}|\bu|^2\bu\,d\mu_i(\bu),\quad i=1, 0.
$$
The differential information given through the restrictions $\dv\bV_j=0$, $\curl\bU_j=0$ in $\Omega$, can also be transformed into information for the underlying Young measure through the classic div-curl lemma. Indeed, it is well-known that
$$
\bU\cdot\bV=t\alpha_1 q_1+(1-t)\alpha_0 q_0,\quad q_i=\int_{\R^N}|\bu|^4\,d\mu_i(\bu),\quad i=1, 0.
$$
We therefore are led to consider the following important problem.

\begin{problem}\label{probuno}
Find a characterization of the set of triplets $(t, \bU, \bV)\in[0, 1]\times\R^N\times\R^N$ so that 
\begin{gather}
\bU=t\bU_1+(1-t)\bU_0,\quad \bV=t\alpha_1\bV_1+(1-t)\alpha_0\bV_0,\nonumber\\
\bU\cdot\bV=t\alpha_1 q_1+(1-t)\alpha_0 q_0,\nonumber
\end{gather}
and there are probability measures $\mu_i$ supported in $\R^N$ in such a way that
$$
\int_{\R^N}(\bu, |\bu|^2\bu, |\bu|^4)\,d\mu_i(\bu)=(\bU_i, \bV_i, q_i),\quad i=1, 0.
$$
\end{problem}
To treat this problem, we are naturally led to call to mind the classical algebraic moment problem.

\section{Moment problems}
A vector $\Phi$ of moments is a finite collections of continuous functions $\Phi=(\phi_i)_i$, with $\phi_1\equiv1$, defined in a euclidean space $\R^N$ (or some other space). The corresponding matrix $\Xi$ of moments is the tensor product $\Xi=\Phi\otimes\Phi$ with entries $\Xi_{ij}=\phi_i\phi_j$. Suppose $\mu$ is a probability measure supported in $\R^N$. The vector $\ba$ and the matrix $\bA$ defined through
\begin{equation}\label{mh}
\ba=\int_{\R^N}\Phi\,d\mu(\bu),\quad \bA=\int_{\R^N}\Phi\otimes\Phi\,d\mu(\bu),
\end{equation}
are the corresponding vector of moments and Hankel matrix, respectively. Note that $\ba$ is the first row and column of $\bA$. 
\begin{proposition}
If $\bA$ and $\ba$ are the vector of moments and Hankel matrix as in (\ref{mh}) with respect to $\Phi$ and a certain probability measure $\mu$, then
\begin{equation}\label{desg}
\bA\ge\ba\otimes\ba
\end{equation}
in the sense of symmetric matrices. 
\end{proposition}
What is quite remarkable is that, sometimes, the condition in this statement is a characterization of feasible Hankel matrices with respect to probability measures. Most of the time it is not, and (\ref{desg}) is just a necessary condition. See \cite{akhiezer}, \cite{shoatamarkin}. 

If we go back to our situation in non-linear conductivity, we take $\Phi=(1, \bu, |\bu|^2)$, and 
$$
\Xi=\Phi\otimes\Phi=\begin{pmatrix}1&\bu&|\bu|^2\\ \bu^T&\bu\otimes\bu& |\bu|^2\bu\\
|\bu|^2&|\bu|^2\bu^T&|\bu|^4\end{pmatrix}.
$$
If $\mu$ is a probability measure supported in $\R^N$, then we write
$$
\int_{\R^N}(1, \bu, |\bu|^2)\otimes(1, \bu, |\bu|^2)\,d\mu(\bu)=
\begin{pmatrix}1&(\bU, U)\\ (\bU, U)^T&\bQ\end{pmatrix}
$$
where $\bU\in\R^N$, $U\in\R$, $\bQ\in\R^{(N+1)\times(N+1)}$. 
If $\mu=\mu_i$ for $i=1, 0$, then we would have
$$
\int_{\R^N}(1, \bu, |\bu|^2)\otimes(1, \bu, |\bu|^2)\,d\mu_i(\bu)=
\begin{pmatrix}1&(\bU_i, U_i)\\ (\bU_i, U_i)^T&\bQ_i\end{pmatrix}.
$$

The conditions on our Problem \ref{probuno} can be recast in the form
\begin{problem}\label{probdos}
Find a characterization of the set of triplets $(t, \bU, \bV)\in[0, 1]\times\R^N\times\R^N$ so that 
\begin{gather}
\bU=t\bU_1+(1-t)\bU_0,\quad \bV=t\alpha_1\bQ_{12, 1}+(1-t)\alpha_0\bQ_{12, 0},\nonumber\\
\bU\cdot\bV=t\alpha_1\bQ_{22, 1}+(1-t)\alpha_0\bQ_{22, 0},\nonumber
\end{gather}
and
\begin{equation*}
\begin{pmatrix}\bQ_{11, i}&\bQ_{12, i}\\ \bQ_{12, i}^T&\bQ_{22, i}\end{pmatrix}\ge \begin{pmatrix} \bU_i&U_i\end{pmatrix}\otimes\begin{pmatrix} \bU_i&U_i\end{pmatrix},\quad U_i\ge|\bU_i|^2,\quad i=1, 0,\nonumber
\end{equation*}
for  some $\bQ_{11,i}\in\R^{N\times N}$, $\bQ_{12,i}, \bU_i\in\R^N$, and $\bQ_{22,i}, U_i\in\R$.
\end{problem}


\section{Necessary conditions}
We can solve explicitly Problem \ref{probdos}. Because the set of triplets in Problem \ref{probdos} is larger than those in Problem \ref{probuno}, this will provide necessary conditions that in general will not be sufficient. 

\begin{theorem}
If $(t, \bU, \bV)$ is such that
$$
\gamma|\bU|^4\le\bV\cdot\bU,
$$
for
$$
\gamma=\frac{\alpha_1\alpha_0}{((1-t)\alpha_1^{1/3}+t\alpha_0^{1/3})^3},
$$
then matrices $\bQ_{11,i}\in\R^{N\times N}$, vectors $\bQ_{12,i}, \bU_i\in\R^N$ and scalars $\bQ_{22,i}, U_i\in\R$ can be found so that the constraints in Problem \ref{probdos} hold. The converse is also correct. 
\end{theorem}
\begin{proof}
Suppose that the triplet $(t, \bU, \bV)$ is given in such a way that 
\begin{equation}\label{condicionimp}
\frac1{\alpha_1}(\bV\cdot\bU-(1-t)c)>|\bU-(1-t)\ba|^4,\quad 
\frac1{\alpha_0}(\bV\cdot\bU+tc)>|\bU+t\ba|^4,
\end{equation}
for a certain vector $\ba$ and scalar $c$. Take then $U_1$ and $U_0$, non-negative, so that
\begin{gather}
|\bU-(1-t)\ba|^4\le U_1^2<\frac1{\alpha_1}(\bV\cdot\bU-(1-t)c),\nonumber\\
|\bU+t\ba|^4\le U_0^2<\frac1{\alpha_0}(\bV\cdot\bU+tc).\nonumber
\end{gather}
Take any two vectors $\bu_1$, $\bu_0$, in such a way that
$$
\bV=(1-t)\alpha_0[U_0(\bU+t\ba)+\bu_0]+t\alpha_1[U_1(\bU-(1-t)\ba)+\bu_1].
$$
Under this identity, there is a vector $\bbe$ such that 
\begin{gather}
\frac1{\alpha_1}(\bV-(1-t)\bbe)=U_1(\bU-(1-t)\ba)+\bu_1,\nonumber\\
\frac1{\alpha_0}(\bV+t\bbe)=U_0(\bU+t\ba)+\bu_0.\nonumber
\end{gather}
Finally, take 
\begin{gather}
\sigma_1=\frac1{\alpha_1}(\bV\cdot\bU-(1-t)c)-U_1^2>0,\quad \sigma_0=\frac1{\alpha_0}(\bV\cdot\bU+tc)-U_0^2>0,\nonumber\\
\sigma_1\gamma_1=|\bu_1|^2,\quad\sigma_0\gamma_0=|\bu_0|^2,\quad \gamma_i>0.\nonumber
\end{gather}
and put
\begin{gather}
\bQ_{11, 1}=(\bU-(1-t)\ba)\otimes (\bU-(1-t)\ba)+\gamma_1\id,\nonumber\\
\bQ_{11, 0}=(\bU+t\ba)\otimes (\bU+t\ba)+\gamma_0\id,\nonumber
\end{gather}
as well as
\begin{gather}
\bU_1=\bU-(1-t)\ba,\quad \bU_0=\bU+t\ba,\quad \ba\in\R^N,\nonumber\\
\alpha_1\bQ_{12,1}=\bV-(1-t)\bbe,\quad \alpha_0\bQ_{12,0}=\bV+t\bbe,\quad \bbe\in\R^N,\nonumber\\
\alpha_1\bQ_{22,1}=\bV\cdot\bU-(1-t)c,\quad \alpha_0\bQ_{22,0}=\bV\cdot\bU+tc,\quad c\in\R.\nonumber
\end{gather}
$\id$ is the identity matrix. 
It is now elementary to check that, with this choice of matrices, vectors, and scalars, all of the conditions
$$
\begin{pmatrix}\bQ_{11, i}&\bQ_{12, i}\\ \bQ_{12, i}^T&\bQ_{22, i}\end{pmatrix}\ge \begin{pmatrix} \bU_i&U_i\end{pmatrix}\otimes\begin{pmatrix} \bU_i&U_i\end{pmatrix},\quad U_i\ge|\bU_i|^2,\quad i=1, 0,\nonumber
$$
are met. Indeed, by the various choices explained above, we have
$$
\begin{pmatrix}\bQ_{11, i}&\bQ_{12, i}\\ \bQ_{12, i}^T&\bQ_{22, i}\end{pmatrix}- \begin{pmatrix} \bU_i&U_i\end{pmatrix}\otimes\begin{pmatrix} \bU_i&U_i\end{pmatrix}=\begin{pmatrix}\gamma_i\id&\bu_i\\ \bu_i^T&\sigma_i
\end{pmatrix},
$$
and this last matrix is positive (non-negative) definite, again by choice of the various parameters. The other inequality $U_i\ge|\bU_i|^2$ also holds by choice of those values. 
If equalities occur in (\ref{condicionimp}), by continuity the conclusion is still valid. 

On the contrary, it is straightforward to have that $\bQ_{22, i}\ge U_i^2\ge|\bU_i|^4$ from the inequalities. Since we also have
\begin{gather}
\bU_1=\bU-(1-t)\ba,\quad \bU_0=\bU+t\ba,\quad\hbox{for some } \ba\in\R^N,\nonumber\\
\alpha_1\bQ_{22,1}=\bV\cdot\bU-(1-t)c,\quad \alpha_0\bQ_{22,0}=\bV\cdot\bU+tc,\quad \hbox{for some } c\in\R,\nonumber
\end{gather}
then  (\ref{condicionimp}) is correct. 

Finally, it is an interesting exercise to conclude that (\ref{condicionimp}) is equivalent to the constraint on $(t, \bU, \bV)$ in the statement. Note that (\ref{condicionimp}) can be rewritten, upon elimination of $c$,  in the form
$$
\bU\cdot\bV\ge t\alpha_1|\bU-(1-t)\ba|^4+(1-t)\alpha_0|\bU+t\ba|^4.
$$
This process follows the same path as with the linear case. If we consider the function 
$$
\Phi(\bbU)=t\alpha_1|\bU-(1-t)\bbU|^4+(1-t)\alpha_0|\bU+t\bbU|^4,
$$
with the only change on the exponent, the minimum is found to be at 
$$
\bbU=\frac{\alpha_1^{1/3}-\alpha_0^{1/3}}{(1-t)\alpha_1^{1/3}+t\alpha_0^{1/3}}\bU.
$$
By substituting this $\bbU$ in $\Phi$, and demanding $\Phi(\bbU)\le\bU\cdot\bV$ we arrive at the inequality in the statement. 
\end{proof}
This proof also yields an important piece of information. If we recall that vector $\bbU$ is somehow related to the direction of lamination, once we have that a triplet $(t, \bU, \bV)$ complies with the inequality in the statement, then feasible $\bbU$'s are those that verify
$$
t\alpha_1|\bU-(1-t)\bbU|^4+(1-t)\alpha_0|\bU+t\bbU|^4\le\bU\cdot\bV.
$$
This is some sort of quartic ellipsoid in $\R^N$. 

\section{Sufficient conditions}
For the sufficiency part, let us start by looking at a single, first-order laminate.
This corresponds to taking $\mu_i$ a delta measure $\mu_i=\delta_{\bU_i}$. Then the probability measure
$$
\mu=t\delta_{(\bU_1, \alpha_1|\bU_1|^2\bU_1)}+(1-t)\delta_{(\bU_0, \alpha_0|\bU_0|^2\bU_0)}
$$
will be a first-order laminate, provided that
\begin{equation}\label{primerorden}
(\alpha_1|\bU_1|^2\bU_1-\alpha_0|\bU_0|^2\bU_0)\cdot(\bU_1-\bU_0)=0.
\end{equation}
If we let, as above, $\bbU=\bU_1-\bU_0$, the direction of lamination, then
$$
\bU_1=\bU-(1-t)\bbU,\quad \bU_0=\bU+t\bbU,
$$
and by taking these expressions back into (\ref{primerorden}), 
$$
(\alpha_1|\bU-(1-t)\bbU|^2(\bU-(1-t)\bbU)-\alpha_0|\bU+t\bbU|^2(\bU+t\bbU))\cdot\bbU=0.
$$
This is a quartic equation for feasible lamination directions $\bbU$ corresponding to first moment $\bU$ with volume fraction $t$. The associated field $\bV$ is given by
$$
\bV=t\alpha_1|\bU-(1-t)\bbU|^2(\bU-(1-t)\bbU)+(1-t)\alpha_0|\bU+t\bbU|^2(\bU+t\bbU).
$$

\begin{lemma}\label{convexa}
The function
$$
\Psi_{(t, \bU)}(\bx)=(\alpha_0|t\bx+\bU|^2(t\bx+\bU)+\alpha_1|(1-t)\bx-\bU|^2((1-t)\bx-\bU))\cdot\bx,\quad \bx\in\R^N
$$
is strictly convex, for arbitrary choices of $\bU\in\R^N$, and $t\in[0, 1]$. 
\end{lemma}
\begin{proof}
Just notice that $\bx=(t\bx+\bU)+((1-t)\bx-\bU)$, and so
$$
\Psi_{(t, \bU)}(\bx)=\alpha_0|t\bx+\bU|^4
+\alpha_1|(1-t)\bx-\bU|^4+(\alpha_1+\alpha_0)((1-t)\bx-\bU)\cdot(t\bx+\bU).
$$
Each one of these three terms is a strictly convex function of $\bx$. 
\end{proof}

Lemma \ref{convexa} ensures that the set $\Psi_{(t, \bU)}(\bbU)\le0$ is a strictly convex set for every possible selection of $\bU$ and $t$. So if $\bbU$ is such that $\Psi_{(t, \bU)}(\bbU)<0$, it can be decomposed as a convex combination (in a non-unique way) $\bbU=r\bbU_1+(1-r)\bbU_0$ for some $r\in(0, 1)$, and in such a way that $\Psi_{(t, \bU)}(\bbU_i)=0$ (for the same values of $\bU$ and $t$). In this way, if we further set
\begin{equation}\label{uvee}
\bV_i=t\alpha_1|\bU-(1-t)\bbU_i|^2(\bU-(1-t)\bbU_i)+(1-t)\alpha_0|\bU+t\bbU_i|^2(\bU+t\bbU_i),\quad i=1, 0,
\end{equation}
then the probability measure 
$$
r\delta_{(\bU, \bV_1)}+(1-r)\delta_{(\bU, \bV_0)}
$$
is a laminate, because trivially $(\bU-\bU)\cdot(\bV_1-\bV_0)=0$, and for
$$
\bU_{(1, i)}=\bU-(1-t)\bbU_i,\quad \bU_{(0, i)}=\bU+t\bbU_i,\quad i=1, 0,
$$
the two probability measures
$$
t\delta_{(\bU_{(1, i)}, \alpha_1|\bU_{(1, i)}|^2\bU_{(1, i)})}+(1-t)\delta_{(\bU_{(0, i)}, \alpha_0|\bU_{(0, i)}|^2\bU_{(0, i)})},
$$
can also be laminated precisely because $\Psi_{(t, \bU)}(\bbU_i)=0$. 
Altogether the probability measure
\begin{gather}
rt\delta_{(\bU_{(1, 1)}, \alpha_1|\bU_{(1, 1)}|^2\bU_{(1, 1)})}+r(1-t)\delta_{(\bU_{(0, 1)}, \alpha_0|\bU_{(0, 1)}|^2\bU_{(0, 1)})}\nonumber\\
+(1-r)t\delta_{(\bU_{(1, 0)}, \alpha_1|\bU_{(1, 0)}|^2\bU_{(1, 0)})}+(1-r)(1-t)\delta_{(\bU_{(0, 0)}, \alpha_0|\bU_{(0, 0)}|^2\bU_{(0, 0)})}\nonumber
\end{gather}
is a second-order laminate for the triplet $(t, \bU, \bV)$ where
$$
\bV=r\bV_1+(1-t)\bV_0,
$$
and $\bV_i$ are given in (\ref{uve}). 
\begin{theorem}
Consider the map
$$
\Phi_{(t, \bU)}:\bx\mapsto \bV=t\alpha_1|\bU-(1-t)\bx|^2(\bU-(1-t)\bx)+(1-t)\alpha_0|\bU+t\bx|^2(\bU+t\bx),
$$
and the set
$$
\bC(t, \bU)=\{\bbU\in\R^N: \Psi_{(t, \bU)}(\bbU)\le0\}.
$$
The triplet $(t, \bU, \bV)$ is reachable by lamination if and only if 
\begin{equation}\label{condlam}
\bV\in\Phi_{(t, \bU)}(\bC(t, \bU)).
\end{equation}
Moreover, when a triplet satisfies this condition, it can always be reached by a second-order laminate. 
\end{theorem}
\begin{proof}
According to our previous discussion, vector $\bV$ should belong to the convex hull of the image under $\Phi_{(t, \bU)}$ of the boundary of $\bC(t, \bU)$, i.e. the set of vectors where $\Psi_{(t, \bU)}=0$. Let us study the structure of the map $\Phi_{(t, \bU)}$ to check that indeed 
\begin{enumerate}
\item $\Phi_{(t, \bU)}(\bC(t, \bU))$ is a convex set;
\item $\partial\Phi_{(t, \bU)}(\bC(t, \bU))=\Phi_{(t, \bU)}(\partial\bC(t, \bU))$.
\end{enumerate}
It is immediate to rewrite 
$$
\Phi_{(t, \bU)}(\bx)=(t\alpha_1|\bU-(1-t)\bx|^2+(1-t)\alpha_0|\bU+t\bx|^2)\bU+t(1-t)(\alpha_0|\bU+t\bx|^2-\alpha_1|\bU-(1-t)\bx|^2)\bx.
$$
For a real constant $k$, the set where 
$$
\alpha_0|\bU+t\bx|^2-\alpha_1|\bU-(1-t)\bx|^2=k
$$
represent, when they are non-empty, concentric circles with a center depending on $\bU$ and the other parameteres of the problem, and a certain radius also depending on these parameters. For points on these sets, the image under $\Phi_{(t, \bU)}$ is given by
$$
\alpha_1|\bU-(1-t)\bx|^2\bU+k(1-t)\bU+t(1-t)k\bx.
$$
Without the first term, it would be a translated homotetic circle. That first term changes the ``center" of the image with $\bx$ according to the factor $|\bU-(1-t)\bx|^2$, but the image will still be a convex, deformed ball. This information, together with the smoothness of the map, implies both claimed properties above. 

Our same discussion prior to the statement of this result, implies that if the triplet $(t, \bU, \bV)$ complies with property (\ref{condlam}) in the statement, 
then it can be obtained by a second-order laminate. On the contrary, the convexity of all sets involved implies that higher-order laminates cannot produced triplets not verifying that same condition. 
\end{proof}

\end{document}